\documentclass[12pt]{amsart}
\pdfoutput=1
\usepackage{amsbsy,amssymb,amscd,amsfonts,latexsym,amstext,delarray,
amsmath,graphicx, xcolor, mathtools, float, tikz, bbm} 
\usetikzlibrary{decorations.pathreplacing,calligraphy, patterns, calc,fit,matrix,arrows,automata,positioning, patterns,decorations.markings, decorations.pathmorphing}
\usepackage[font={small, sc}]{caption}
\usepackage{color}

\newtheorem{thm}{Theorem}[section]
\newtheorem{prop}[thm]{Proposition}

\newtheorem{rem}[thm]{Remark}

\numberwithin{equation}{section}

\def\bC{{\mathbb C}}

\def\bG{{\mathbb G}}
\def\bH{{\mathbb H}}

\def\bN{{\mathbb N}}

\def\bP{{\mathbb P}}
\def\bQ{{\mathbb Q}}
\def\bR{{\mathbb R}}

\def\bZ{{\mathbb Z}}

\def\GL{{\rm GL}}

\def\PSL{{\rm PSL}}

\def\SL{{\rm SL}}

\def\Cstar{{\mathcal{C}^*}}

\title[Higher-weight limiting modular symbols]{Higher-weight limiting modular symbols}
\author{Jane Panangaden}
\date{2023}
\address{Mathematics Department, Pitzer College, Claremont \\ USA}
\email{panangaden@pitzer.edu}

\begin{document}
\begin{abstract} We begin with the higher-weight modular symbols introduced by Shokurov, which generalize Manin's weight-$2$ modular symbols. We then define higher-weight limiting modular symbols associated to vertical geodesics with one endpoint at an irrational real number, by means of a limiting procedure on Shokurov's modular symbols. These are analogous to the Manin-Marcolli limiting modular symbols for the weight-$2$ case, which are given by a similar limiting procedure on the Manin modular symbols. We show that the limit defining the higher-weight limiting modular symbol is equivalent everywhere to a limit given by approximating the irrational endpoint by its continued fraction expansion. This is done by means of shifting to a coding space, as in the approach of Kesseb\"ohmer and Stratmann in the weight-$2$ case. 

\end{abstract}

\maketitle

\section{Introduction}

There is a rich interplay between number theory and quantum statistical mechanics, beginning with the construction of the Bost-Connes system \cite{BC}. This system is a $\Cstar$-dynamical system that has a geometric interpretation as a coordinate algebra of $1$-dimensional $\bQ$-lattices up to a commensurability equivalence relation. It is related to the explicit class field theory of $\bQ$ in the following sense. The system has ground states ($0$-temperature equilibrium states) which are parameterized by $\GL_1(\hat{\bZ})$. When evaluated on points in a certain  arithmetic subalgebra, the ground states yield a set of algebraic numbers which generate the maximal abelian extension $\bQ^{cycl}$. This picture was later extended by the Connes-Marcolli $\GL_2$-system \cite{CoMa2}, which has an analogous geometric interpretation in terms of $2$-dimensional $\bQ$-lattices. This system also has ground states, which are parameterized by the invertible $2$-d $\bQ$-lattices, or equivalently by $\SL_2(\bZ) \backslash (\GL_2(\hat{\bZ}) \times \bH)$. A certain arithmetic subalgebra was also constructed. The $\GL_2$ system is related to the explicit class field theory of the imaginary quadratic extensions in the following way. The ground state that corresponds to an invertible lattice given by $(\rho, \tau) \in  \GL_2(\hat{\bZ}) \times \bH$, when evaluated on points in the arithmetic subalgebra, generates the specialization at $\tau$ of the modular field. 
\\
\\
Manin's real multiplication program (\cite{Man}, \cite{Man2}) suggests an approach of viewing $\bR$ as an ``invisible" boundary of $\bH$, where points in $\bR$ are viewed as noncommutative degenerations of complex tori. In joint work with Marcolli \cite{MarPan}, we use this idea to construct a boundary version of the $\GL_2$-system, where the boundary $\bP^1(\bR)$ is incorporated directly with the action of a shift operator which implements the shift on the continued fraction expansion. We obtain a family of ground states for the boundary $\GL_2$-system parameterized by 
\[ \GL_2(\bZ) \backslash (\GL_2(\hat{\bZ}) \times  \mathcal{\tilde{P}} ) \times \mathcal{D}_{[0,1] \cap \bQ}  \]
where $ \mathcal{\tilde{P}}$ is a discrete space built from cosets $\bP_{\alpha} = \Gamma \alpha \Gamma / G$, $\Gamma$ is a subgroup of $\GL_2(\bQ)$ and $G$ is a finite index subgroup of $\Gamma$, and where $\mathcal{D}_{[0,1] \cap \bQ}$ is a disconnection algebra in the sense of \cite{Spi}.  We also construct a certain arithmetic subalgebra associated to the boundary-$\GL_2$ system. The evaluation of the ground states on elements of the arithmetic algebra yields pairings of weight-$2$ cusp forms with the limiting modular symbols introduced by Manin and Mariolli in \cite{ManMar}. 
\\
\\
The modular symbol was introduced by Manin in \cite{Man-modsymb}. For a modular group $G$, a modular symbol $\{ \alpha, \beta \}_G$ associated to a pair of cusps $\alpha, \beta \in \bP^1(\bQ)$ is an element of the homology group $H_1(X_G, \bR)$ where $X_G$ is the modular curve. There is a perfect pairing between modular symbols and $S_2(G)$, the space of weight-$2$ cusp forms. The limiting modular symbols of \cite{ManMar} were introduced in order to extend the picture of modular symbols to the invisible boundary of modular curves. For a cusp $\beta \in \bP^1(\bR) \setminus \bP^1(\bQ)$, the limiting modular symbol $\{ \{ \star, \beta \} \}_G$ is also an element of the homology group $H_1(X_G, \bR)$, obtained from the modular symbols via a limiting procedure. The limiting modular symbols are known to exist almost everywhere. They can be expressed in terms of continued fraction expansions and in particular they are non-vanishing at real quadratic points, which have periodic continued fraction expansion. 
\\
\\
In \cite{Shok}, Shokurov introduced modular symbols of weight $w+2$. The starting point for this construction is to take the Kuga modular variety $B_G^w$, which is the Kuga variety of the elliptic surface $B_G$ over the modular curve $\overline{X}_G$ with projection 
\[ \Phi : B_G \rightarrow \overline{X}_G.\]
 The weight-$w+2$ modular symbol $\{\alpha, \beta, N,M\}_G$ where $\alpha, \beta \in \bQ \cup \{ i \infty \}$  and $N,M \in \bZ^w $ is an element of the homology group
\[ H_1(\overline{X}_G, \Pi, (\bG \otimes_{\bQ} \bQ)^w ) \]
where $\Pi$ is the set of cusps, $\bG$ is the homological invariant of $B_G$, and $(\bG \otimes_{\bQ} \bQ)^w$ is the symmetric tensor power. Shokurov showed in the same paper that there is a non-degenerate pairing between these weight $w+2$ modular symbols and $S_{w+2}(G) \oplus \overline{S_{w+2}(G)}$ where $S_{w+2}(G)$ is the space of weight $w+2$ cusp forms on $X_G$. 
\\
\\
It is expected that both the limiting modular symbols and the $\GL_2$ and boundary-$\GL_2$-systems will generalize to a higher-weight setting. Then, the evaluations of the ground states of these higher-weight systems on their corresponding arithmetic subalgebras should give a pairing of the higher-weight limiting modular symbols with higher weight cusp forms, and it is expected that the relations between periods of Hecke eigencuspforms of \cite{Man-Hecke} will arise. The purpose of this note is to present a preliminary result in this direction, namely that the higher-weight limiting modular symbols can be defined via a limiting procedure on Shokurov's higher-weight modular symbols. 
\\
\\
More precisely, we first move to a coding space $\Sigma_G$ which describes geodesics in the hyperbolic half plane by a sequence of ``type changes" as they pass through tiles of the Farrey tessellation. It can be shown that this coding space is isomorphic to a set of sequences $( (x_k,e_k) )_k$ with $x_k  \in  \bZ^{\times}$ and $e_k \in E_G $
where $E_G$ is a fixed set of representative elements of the cosets $G \backslash \PSL_2(\bZ)$ and satisfying the further constraints that for all $k$, $x_k x_{k+1} < 0$ and $e_{k+1} = \tau_{x_k} e_k$ where 
\[ \tau_{x_k} (e_k) \equiv_G e_k ST^{x_k} \]
and 
\[ S = \begin{pmatrix} 0 & -1 \\ 1 & 0 \end{pmatrix} \;\;\;\; \textrm{and} \;\;\;\; T = \begin{pmatrix} 1 & 1 \\ 0 & 1 \end{pmatrix} \]
generate $\PSL_2(\bZ)$. 
\\
\\
Kesseb\"ohmer and Stratmann have shown in \cite{KS} that in the coding space, the weight-$2$ limiting modular symbol corresponding to the geodesic coded by $((x_k, e_k))_k$ is
\[  \lim_{t \rightarrow \infty} \frac{1}{t} \{i, e_1(x + i \exp(-t))\}_G \]
where $x$ is given by the continued fraction expansion  $-\textrm{sign}(x_1) [|x_1|, |x_2|, ... ]$. Furthermore, they have shown that this is everywhere equal to the limit
\[ \lim_{n \rightarrow \infty} \frac{1}{2 \log q_n(|x|)} \sum_{k=1}^{n} \{ e_k (i\infty), e_k(0) \}_G \]
that arises from approximating the endpoint of the geodesic by its continued fraction expansion. (Recall that $\frac{p_k(x)}{q_k(x)}$ is the $k^{th}$ continued fraction approximant of $x$.)
\\
\\
Our main result (Theorem \ref{mainthm}) is the analogous statement, but for the higher-weight modular symbols. We first move to the coding space and then define the higher-weight limiting modular symbol for $((x_k,e_k))_k$ by
\[ \lim_{t \rightarrow \infty} \frac{1}{t} \{i, e_1(x + i \exp(-t)), N, M\}_G \]
and show that this is equal to the limit 
\[  \lim_{n \rightarrow \infty} \frac{1}{2 \log q_n(|x|)} \sum_{k=1}^{n} \{ e_k (i\infty), e_k(0), \tilde{g}_{k-1}^{-1}(x) \cdot (N,M) \}_G \]
where, as before, $x$  has continued fraction expansion $-\textrm{sign}(x_1) [|x_1|, |x_2|, ... ]$, $\tilde{g}_{k-1}(x) = e_1 \overline{g_{k-1}}(x) e_k^{-1}$, and 
\[ g_k(x) = \begin{pmatrix}p_{k-1}(x) & p_k(x) \\ q_{k-1}(x) & q_k(x) \end{pmatrix} \] 
acts by linear fractional transformation.
\\
\\
This is a preliminary result. There is much more work to be done in developing a higher-weight version of the boundary-$\GL_2$-system, describing its ground states, and studying the evaluations of the ground states on points in an appropriate arithmetic algebra. These evaluations are expected to yield some relations involving the higher weight limiting modular symbols and higher weight cusp forms.

\section{Modular symbols} 

\subsection{Weight-2 modular symbols} 

We begin by reviewing the definition and important properties of the modular symbols of weight $2$ introduced in \cite{Man-modsymb}. We fix some modular curve $X_G$ for a modular group $G$. The modular symbol \index{modular symbol} associated to points $\alpha$, $\beta$ in $P^1(\bQ)$ is a real homology class in $H_1(X_G,\bR)$, constructed as follows. Consider $C_{\alpha, \beta}$ the oriented geodesic going from $\alpha$ to $\beta$ in $\bH$. Let $\varphi : \bH \cup P^1(\bQ) \rightarrow X_G$ be the quotient map. Because $\alpha$ and $\beta$ are cusps, the image $\varphi(C_{\alpha, \beta})$ is closed on $X_G$. We defined the modular symbol $\{\alpha, \beta\}_G$ by 
\[\int_{ \{\alpha, \beta\}_G} \omega := \int_ {\alpha}^{\beta} \varphi^*(\omega) = \int_{\varphi(C_{\alpha,\beta})}\omega\]
for $\omega$ a differential form on $X_G$. 
\\
\\
The modular symbols are related to the weight-$2$ cusp forms as follows. Let $S_2(G)$ be the space of weight-$2$ cusp forms and fix $f \in S_2(G)$. The function $f$ doesn't descend to a function on $X_G$ because it isn't $G$-invariant, but the one-form $fdz$ does. We have the invariance, 
\[ f(\gamma \cdot z) d(\gamma \cdot z) =  f\left(\frac{az+b}{cz+d}\right) d\left(\frac{az+b}{cz+d}\right) = (cz+d)^2 f(z) \frac{ac-bd}{(cz+d)^2}dz = f(z)dz \]
for $\gamma = \begin{pmatrix} a&b \\c&d\end{pmatrix}\in G \subset \SL_2(\bZ)$. We then obtain a pairing 
\[ \langle, \rangle : S_2(G) \times H_1(X_G, \bZ) \rightarrow \bC\]
 by integrating along the image in $X_G$ of the geodesic in $\bH$ connecting $\alpha$ and $\beta$ 
\[ \langle f, \{ \alpha, \beta \}_G \rangle = \int_{\alpha}^{\beta} f(z)dz.  \]
We extend the pairing to a pairing $\langle, \rangle : S_2(G) \times H_1(X_G, \bR) \rightarrow \bC$ by linearity. This pairing is perfect and it identifies the dual $S_2(G)^*$ with $H_1(X_G, \bZ)$. 
\\
\\
The modular symbols have several basic properties, which all follow easily from the definition:
\begin{enumerate}
\item $\{\alpha, \beta\}_G = - \{\beta, \alpha\}_G$
\item $\{\alpha, \beta\}_G = \{\alpha, \gamma\}_G + \{\gamma, \beta\}_G$
\item $ \{g\alpha, g\beta\}_G = \{\alpha, \beta\}_G$ for all $g\in G$ 
\end{enumerate} 
Because of the second property, it suffices to consider modular symbols of the form $\{0, \alpha\}_G$. We may also decompose modular symbols of this form using the continued fraction expansion of $\alpha$. Let $\alpha = [a_0; a_1, a_2, a_3...]$ be the continued fraction expansion of $\alpha$ and $\frac{p_k(\alpha)}{q_k(\alpha)}$ be the $k^{th}$ continued fraction approximant of $\alpha$ and observe that the matrix 
\[ g_k(\alpha)= \begin{pmatrix} p_k(\alpha) & p_{k-1}(\alpha) \\ q_k(\alpha) & q_{k-1}(\alpha) \end{pmatrix} \]
is in $\GL_2(\bZ)$ as a consequence of the recurrence relations 
\begin{align*}
p_k &= a_k p_{k-1} + p_{k-2} \\
q_k &= a_k q_{k-1} + q_{k-2} 
\end{align*}
where $p_0=0$ and $q_0=1$. 
\\
\\
For $\alpha = [a_1,a_2,..., a_n]$ rational, we can write the modular symbol as a finite sum:
\[ \{0, \alpha \}_G = \sum_{k=1}^n \left\{\frac{p_{k-1}}{q_{k-1}}, \frac{p_k}{q_k} \right\}_G = \sum_{k=1}^n \left\{g_k(0), g_k(i\infty)\right\}_G.\]
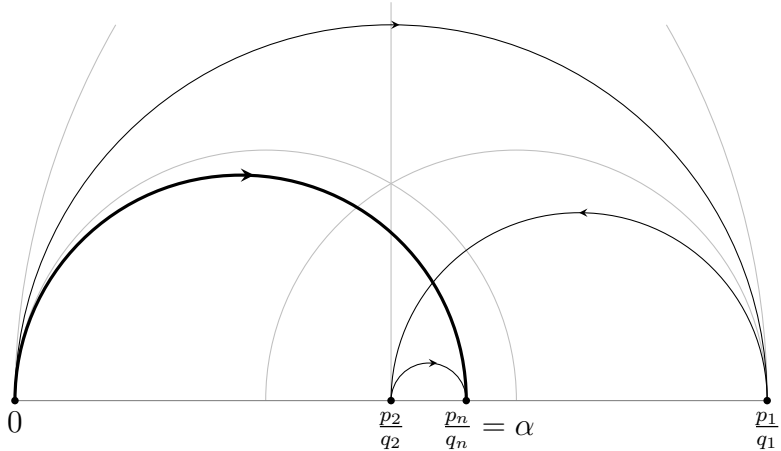
\begin{figure}[H]
\centering
\begin{tikzpicture}[scale=10,>=stealth]
\draw [color=lightgray](1,0) arc (0:30:1cm) ;
\draw [color=lightgray](2/3,0) arc (0:180:1/3) ;
\draw [color=lightgray](1,0) arc (0:180:1/3) ;
\draw [color=lightgray](0,0) arc (0:-30:-1cm) ;
\draw [color=lightgray](1/2,0) -- (1/2,0.53);
\draw [color=gray](0,0) -- (1,0);

\node at (0,0) [circle,fill,inner sep=1pt]{};
\node at (0,0) [below]{$0$};
\node at (1,0) [circle,fill,inner sep=1pt]{};
\node at (1,0) [below]{$\frac{p_1}{q_1}$};
\node at (1/2,0) [circle,fill,inner sep=1pt]{};
\node at (1/2,0) [below]{$\frac{p_2}{q_2}$};
\node at (6/10,0) [circle,fill,inner sep=1pt]{};
\node at (6/10,0) [below]{$\;\; \;\;\; \frac{p_n}{q_n}=\alpha$};

\draw[decoration={
    markings,
    mark=at position 0.5 with {\arrow{<}}}, postaction={decorate}](1,0) arc (0:180:0.5);
\draw[decoration={
    markings,
    mark=at position 0.5 with {\arrow{>}}}, postaction={decorate}](1,0) arc (0:180:0.25);
\draw[decoration={
    markings,
    mark=at position 0.5 with {\arrow{<}}}, postaction={decorate}](0.6,0) arc (0:180:0.05);
\draw[decoration={
    markings,
    mark=at position 0.5 with {\arrow{<}}}, postaction={decorate}, very thick](0.6,0) arc (0:180:0.3);
\end{tikzpicture}
\caption{Approximating a path in $\bH$ for a modular symbol by continued fractions}
\end{figure}

\subsection{Higher-weight modular symbols} 
In \cite{Shok}, Shokurov defined the modular symbols of weight greater than $2$ in the following way. From a pair $(G, w)$ where $G$ is a modular group and $w$ is a weight, one can construct a nonsingular projective variety $B_G^w$ over the complex numbers called a Kuga modular variety. \cite{Shok1} This variety is related to an elliptic surface $B_G$ over the modular curve $\overline{X_G}$. There is a natural projection from $B_G^w$ onto the modular curve $ \Phi^w : B_G^w \rightarrow \overline{X_G}$.
\\
\\
The Kuga modular variety is constructed using as a starting point the modular elliptic surface, which is an elliptic surface $B_G$ over the modular curve $\Phi: B_G \rightarrow \overline{X}_G$. It has the important property that the functional invariant is given by $J_G$, where $J_G$ is a meromorphic function on $\overline{X}_G$ given by the composition of the morphism 
\[ \overline{X}_G \rightarrow \overline{X}_{\SL_2(\bZ)}\]
induced by the subgroup structure $G \subset \SL_2(\bZ)$ with the absolute invariant function
\[ j: \overline{X}_{\SL_2(\bZ)} \rightarrow \bC\] 
extending the standard $j$-invariant. Such an elliptic modular surface is canonically defined in the case that $-1 \notin G$ by \cite{Shioda}. In the absence of this condition, a non-canonical construction with the desired property is given in \cite{Shok1}. 
\\
\\
The Kuga modular variety is obtained by taking the Kuga variety, which can be constructed from any non-singular projective surface over a modular curve, of the modular elliptic surface. We give a very brief sketch of this construction. For details, please see \cite{Shok1}. 
\\
\\
Let $\Delta'$ be the set of non-singular points of $\overline{X}_G$, and $\mathcal{U}'$ be its universal cover. There is an action of 
\[ \mathcal{G}^w = \pi_1(\Delta') \times \bZ^w \times \bZ^w \]
on $\mathcal{U}' \times \bC^w$ given by 
\[(\beta, n, m) (u, \xi) = (\beta u, f_{\beta}(u) (\xi + z(u)n + m) ) \]
where $z$ is a multivalued function on $\Delta'$ defined by 
\[ j(z(u)) = J_G(u)\]
and a choice of branch gives a function $z: \mathcal{U}' \rightarrow \bH$, and 
\[ f_{\beta}(u) = (cz(u) + d)^{-1}\]
where $c,d$ are given by the entries of the matrix $S(\beta)= \begin{pmatrix} a& b \\ c&d \end{pmatrix}$ where $S: \pi_1(\Delta') \rightarrow \SL_2(\bZ)$ is a certain representation of the fundamental group. 
\\
\\
We define 
\[ \overline{B}_G^w |_{\Delta'} = \mathcal{G}^w \backslash (\mathcal{U}' \times \bC^w). \]
By compactifying and resolving singularities, we then obtain a non-singular projective variety $B_G^w$, with a canonical projection $\Phi^w : B_G^w \rightarrow \overline{X}_G$. 
\\
\\
To define the modular symbol of weight $w+2$, we first define $\{\alpha,n,m\}_G$, a boundary modular symbol of weight $w+2$, by a mapping 
\begin{align*}
\{ , , \}_G : \tilde{\bQ} \times \bZ^w \times \bZ^w & \rightarrow H_0(\Pi, (R_1\Phi_*\bQ)^w) \\
(\alpha, n, m) & \mapsto \{\alpha,n,m\}_G 
\end{align*}
where $\tilde{\bQ} = \bQ \cup \{i\infty\}$, $\Pi$ are the cusps, and $(R_1\Phi_*\bQ)^w$ is a symmetric tensor power of the sheaf $R_1\Phi_*\bQ = \mathbb{G} \otimes_{\bQ} \bQ$, where $\mathbb{G}$ is the homological invariant of $B_G$.  
\begin{rem}
In general the sheaf $R_j\Phi_*^w\bQ$ is defined by taking the sheaf of local coefficients 
\[\cup_{v \in \Delta'} H_j(B_v^w, \bQ) \]
and extending it over $\overline{X_G}$. We will only need to use the case $R_1\Phi_{*1}\bQ = \mathbb{G} \otimes_{\bQ} \bQ$, which can be interpreted as a rational homological invariant. 
\end{rem}
This mapping is described carefully in Section 1.1 of \cite{Shok}, but we summarize the construction here. Let $\alpha \in \tilde{\bQ}$, and $n,m \in \bZ^w$. Let $p_0 \in \Pi$ be the cusp corresponding to $\alpha$. There is a decomposition 
\[H_0(\Pi, (R_1\Phi_*\bQ)^w)  = \bigoplus_{p \in \Pi} H_0(p, (R_1\Phi_*\bQ)^w).  \]
The modular symbol $\{\alpha,n,m\}$ is trivial on $H_0(p, (R_1\Phi_*\bQ)^w)$ when $p\neq p_0$, and so will be defined by an element in $H_0(p_0, (R_1\Phi_*\bQ)^w)$. Let $E \subset \overline{X}_G$ be a small disc around $p_0$. Let $U_{\alpha}$ be a neighborhood of $\alpha$ in $\bH' = \bH \setminus \SL_2(\bZ) \{ e^{\frac{2\pi i}{3}} \}$ that covers E, and let $\tilde{\Gamma}_{\alpha}: U_{\alpha} \rightarrow E$ be the covering. Choose a point $z_E \ \in U_{\alpha}$ and let $v_E = \tilde{\Gamma}_{\alpha}(z_E)$. View $z_E$ as a $0$-cell in a cellular decomposition of the disc $E$. Now we let $\{z_E,n,m\}_G^E \subset H_0(E, (R_1\Phi_{*}\bQ)^w)$ the cohomology class of the cycle 
\[ \sum_{j=1}^w (n_je_1 + m_je_2)v_E\]
where $\{e_1,e_2\}$ is a certain basis which we will not describe in detail here. There is a projective system of spaces $H_0(E, (R_1\Phi_*\bQ)^w))$ by morphisms 
\[ H_0(E', (R_1\Phi_*\bQ)^w)) \rightarrow H_0(E, (R_1\Phi_*\bQ)^w))\]
where $E' \subset E \subset \overline{X}_G$ are nested small discs. Finally, we set 
\[ \{\alpha, n,m\}_G = \lim_{\substack{\leftarrow\\E}} \{z_e, n,m\}_G^E.\]
It requires some argument to see that this definition makes sense, but we do not include it here as we will not need to work with this definition directly. 
\\
\\
The modular symbol,  $\{ \alpha, \beta, n, m \}_G $, is then defined (\cite{Shok} Lemma 1.2) via the unique mapping 
\begin{align*}
\tilde{Q}\times \tilde{Q} \times \bZ^w \times \bZ^w & \rightarrow H_1(\overline{X_G}, \Pi, (R_1\Phi_*\bQ)^w) \\
(\alpha, \beta, n, m) & \mapsto \{ \alpha, \beta, n, m \}_G
\end{align*}
such that 
\begin{enumerate}
\item $\partial \{\alpha, \beta, n, m \}_G = \{ \beta,n,m\}_G - \{\alpha, n, m\}_G $ where $\partial$ is the boundary mapping of the pair $(\overline{X_G}, \Pi)$. 
\item For any cusp forms $\Psi_1, \Psi_2 \in S_{w+2}(G)$ 
\[ \langle \{\alpha, \beta,n,m\}_G, (\Psi_1, \overline{\Psi_2}) \rangle = \int_{\alpha}^{\beta} \Psi_1 \Pi_{j=1}^w (n_jz + m_j)dz + \int_{\alpha}^{\beta}\overline{\Psi_2} \Pi_{j=1}^w (n_j\overline{z} + m_j) d\overline{z}\]
where $n = (n_1,...,n_w)$, $m=(m_1,...m_w)$ and $\langle,\rangle$ is the canonical pairing described in \cite{Shok3}:
\[ \langle,\rangle: H_1(\overline{X_G}, Y, (R_1\Phi_*\bQ)^w) \times S_{w+2}(G) \oplus \overline{S_{w+2}(G)} \rightarrow \bC\]
where $Y \subset \overline{X_G}$. 
\end{enumerate}
Importantly, the pairing $\langle,\rangle$ is non-degenerate on $H_1(\overline{X_G}, (R_1\Phi_*\bQ)^w) \times S_{w+2}(G) \oplus \overline{S_{w+2}(G)}$ \cite{Shok3}. 
\\
\\
The modular symbols of higher weight have a similar additivity property to the weight-$2$ case: 
 \[ \{ \alpha, \beta, n,m \}_G + \{\beta, \gamma, n,m\}_G = \{ \alpha, \gamma, n,m\}_G \]
 and they transform by elements
 $g = \begin{pmatrix} a & b \\ c & d \end{pmatrix} \in \GL_2^+({\bZ})$ as 
 \[ g | \{\alpha,\beta,n,m\}_G = \{ g(\alpha), g(\beta), g \cdot (n,m) \}_G = \{ g(\alpha), g(\beta), dn-cm, -bn+am\}_G. \]
 Note that this does not give an action directly on $H_1(\overline{X_G}, \Pi, (R_1\Phi_*\bQ)^w)$, but rather on representations of homology classes as modular symbols. For $g \in G$, we have 
 \begin{equation} \label{Ginv} 
 g |  \{\alpha,\beta,n,m\}_G = \{ g(\alpha), g(\beta), dn-cm, -bn+am\}_G = \{ \alpha, \beta, n, m \}_G. 
 \end{equation}
 Again, due to the additivity property, it is sufficient to consider modular symbols of the form, for $\alpha \in \bQ$
 \[ \{ 0, \alpha,n, m\}_G = - \sum_{k=1}^N \{g_k(0), g_k(i\infty), n,m \}_G. \]

\section{Limiting modular symbols}
Manin and Marcolli \cite{ManMar} introduced a generalization of the modular symbols to the whole boundary $\bP^1(\bR)$ by considering an infinite geodesic $\gamma_{\beta}$ in $\bH$ with one with one end at $\beta \in \bR \diagdown \bQ$ and the other end at $\alpha \in \bR$. Let $x_0 \in \bH$ be a fixed point on $\gamma_{\beta}$ and $y(\tau)$ a point along $\gamma_{\beta}$ with an arc length distance of $\tau$ away from $x_0$ towards $\beta$. The \textit{limiting modular symbol} \index{limiting modular symbol} is defined as the following limit, whenever it exists: 
 \begin{equation}\label{limitingmodularsymbol} 
  \{\{ * , \beta\}\}_G = \lim_{\tau \rightarrow \infty} \frac{1}{\tau} \{x_0, y(\tau) \}_G \in H_1(X_G, \bR) 
  \end{equation}
  where $\{x_0, y(\tau)\}_G$ is the homology class determined by the geodesic arc between $x_0$ and $y(\tau)$ in $\bH$. The limit is independent of the choice of $x_0$ and of $\gamma_{\beta}$ (\S 2 of \cite{ManMar}).
  
 \subsection{Shift map and the Lyapunov spectrum}
 To study the weight-$2$ limiting modular symbols, we consider a modular curve of the form $X_G = \textrm{PGL}_2(\bZ) \backslash (\bH \times \mathbb{P})$ where $\mathbb{P} = \textrm{PGL}_2(\bZ)/G$ and the associated shift map 
 \begin{equation}\label{Tshiftmap2}
 \begin{aligned}
  T: [0,1] \times \mathbb{P}& \rightarrow [0,1] \times \mathbb{P} \\
  (\beta, t) &\mapsto \left( \frac{1}{\beta} - \left[ \frac{1}{\beta} \right] , \begin{pmatrix}-[1/\beta] & 1 \\ 1 & 0 \end{pmatrix} t\right).
  \end{aligned}
  \end{equation}
  Defining a map $\phi: \mathbb{P} \rightarrow H_1(\overline{X_G}, \Pi, \bR)$ by 
  \[ \phi(s) = \{g(0), g(i\infty)\}_G\]
  where $g \in \PSL_2(\bZ)$ is a representative of the coset $s \in \mathbb{P}$, we see that $g_k$ acts on points $(\beta, t) \in [0,1]\times\mathbb{P}$ as the $k^{th}$ power of the shift operator $T$. Precisely, 
  \[ \phi(T^k(\beta,t)) = \{g_k(\beta)(0), g_k(\beta)(i\infty)\}_G = - \left\{ \frac{p_{k-1}(\beta)}{q_{k-1}(\beta)}, \frac{p_k(\beta)}{q_k(\beta)} \right\}_G \]
  where, as before, 
   \[ g_k(\beta) = \begin{pmatrix}p_{k-1}(\beta) & p_k(\beta) \\ q_{k-1}(\beta) & q_k(\beta) \end{pmatrix} \] 
   acts by Mobius transformations. 
   \newline
   \newline
It is shown in \cite{Mar} that the limiting modular symbol can be computed on certain level sets as a Birkhoff average. The level sets are given by the Lyapunov spectrum of the shift map on the unit interval 
\begin{equation} \label{Tshift}
\begin{aligned}
T: [0,1]& \rightarrow [0,1] \\
\beta &\mapsto \frac{1}{\beta} - \left[ \frac{1}{\beta} \right]. 
\end{aligned}
\end{equation}
Recall that the Lyapunov exponent \index{Lyapunov exponent} of a map $T: [0,1] \rightarrow [0,1]$ is given by the $T$-invariant function
\[\lambda(\beta) = \lim_{n \rightarrow \infty} \frac{1}{n} \log | (T^n)'(\beta)|. \]
In the particular case of $T$ defined by equation \ref{Tshift}, the Lyapunov exponent is 
\begin{equation} \label{Lyapunov}
\lambda(\beta) = 2 \lim_{n \rightarrow \infty} \frac{1}{n} \log q_n(\beta). 
\end{equation}
A theorem of L\'{e}vy \cite{Levy} shows that $\lambda(\beta) = \frac{\pi^2}{6\log2}$ for almost all $\beta$. We can decompose the unit interval into level sets of \ref{Lyapunov}, $L_c = \{ \beta \in [0,1] : \lambda(\beta) = c\}$
\[[0,1] = \cup_{c \in \bR} L_c \cup \{\beta \in [0,1] : \lambda(\beta) \textrm{ does not exist} \}. \]
Then, we have the following result about the limiting modular symbols. 
\begin{prop}[\cite{Mar} Theorem 2.1] 
For a fixed $c \in \bR$ and for $\beta \in L_c$, the limiting modular symbol \ref{limitingmodularsymbol} is computed by 
\begin{equation}\label{Marlimit}
\lim_{n \rightarrow \infty} \frac{1}{cn} \sum_{k=1}^n \phi \circ T^k(\beta, t_0) 
\end{equation}
where $T$ is the shift operator defined in \ref{Tshiftmap2} and $t_0$ is a base point. 
\end{prop}
It is easy to check that the shift of the continued fraction expansion is measure-preserving with respect to the Gauss measure 
\[ d\mu = (\log2)^{-1} \frac {dx}{1+x}\]
and so the limiting modular symbol exists almost everywhere. However, it is also known that there is an exceptional set of measure $0$ and Hausdorff dimension $1$ where $\lambda(\beta)$ does not exist (\cite{PoWe} Theorem 3). On the exceptional set, the limiting modular symbol cannot be written as the limit \ref{Marlimit}. Finally, in the special case that $\beta$ is a quadratic irrationality (and hence has a periodic continued fraction expansion) it is shown (\cite{Mar} Lemma 2.2) that the limiting modular symbol is given by 
\[\{\{*, \beta\}\}_G = \frac{\sum_{k=1}^{n} \{g_k^{-1}(\beta) \cdot g(0), g_k^{-1}(\beta) \cdot g(i\infty)\}_G}{\lambda(\beta)n} \]
where $n$ is the period of the continued fraction expansion. In this case it is also known that the limit $\lambda(\beta)$ converges to a positive finite number, so that in particular the limiting modular symbol does not vanish. 

\subsection{Higher-weight limiting modular symbol} 

  To extend this picture to the higher weight setting, we now define $\phi: \mathbb{P} \times \bZ \times \bZ \rightarrow H_1(\overline{X_G}, \Pi, (R_1\Phi_*\bQ)^w)$ by 
  \[ \phi(s,n,m) = g | \{0,i \infty,n,m\}_G = \{ g(0), g(\infty), dn-cm, -bn+am\}_G\]
  where $g^{-1} = \begin{pmatrix} a & b \\ c & d \end{pmatrix} \in \PSL_2({\bZ})$ and $g$ is a representative of the coset $s \in \mathbb{P}$. The action of the shift operator on the higher-weight modular symbols is now described by the relation
   \begin{align}\label{nmaction}
    \phi(T^k(\beta,t),n,m) &=  \{g_k(0), g_k(i\infty), g_k^{-1} \cdot (n,m) \}_G \nonumber \\ 
    & = - \left\{ \frac{p_{k-1}(\beta)}{q_{k-1}(\beta)}, \frac{p_k(\beta)}{q_k(\beta)} ,  \begin{pmatrix}0 & -1 \\ -1 & -a_k \end{pmatrix} \dots \begin{pmatrix}0 & -1 \\ -1 & -a_1 \end{pmatrix} \begin{pmatrix}n \\ m \end{pmatrix}\right\}_G 
    \end{align}
where  $\beta = [ a_1, ..., a_N ]$ is the continued fraction expansion of $\beta$. Note that, again, this action is not on $H_1(\overline{X_G}, \Pi, (R_1\Phi_*\bQ)^w)$, but on representations of the homology classes as modular symbols.
\\
\\
Instead of proceeding with this setting, however, we will move to a related setting where we code each geodesic in the hyperbolic plane using cells of the Farey tessellation. It was introduced by Kessenb\'omer and Stratmann in \cite{KS} in order to obtain a more complete description of the standard modular symbols and their level set structure. 

 \subsection{Twisted continued fraction coding and shift space}
 Following \cite{KS} we define a code space related to the dynamical system given by the shift map in the previous section. We recall that an oriented geodesic in $\bH$ can be coded by a sequence of ``type changes''. Consider the Farey tessellation of $\bH$ formed by $\PSL_2(\bZ)$-translates of the triangle with vertices at $0$,$1$, and $i\infty$. As we travel along a geodesic in the positive direction, each tile is intersected in such a way that one vertex of the triangle is on one side, and two vertices of the triangle are on the other. If the single vertex is on the left, we say the visit to the tile is of type $L$, and if the single vertex is on the right, we say it is of type $R$. Let $l = (l_+,l_-)$ be the oriented geodesic with start point $l_+$ and end point $l_-$ and consider the set
\[ \mathcal{L} = \{ l = (l_-,l_+) | 0 < |l_+| \leq 1 \leq |l_-| , l_-l_+ < 0, \textrm{and }l_-,l_+ \in \bR \diagdown \bQ  \} .\]
Each $l \in \mathcal{L}$ is coded by the types of its visits
\begin{align*}
...L^{n_{-2}}R^{n_{-1}}y_lL^{n_1}R^{n_2}... & \textrm{ if } l_- \geq 1 \\
...R^{n_{-2}}L^{n_{-1}}y_lR^{n_1}L^{n_2}...& \textrm{ if } l_- \leq -1
\end{align*}
where $y_l$ is the point where $l$ intersects the imaginary axis. 

\begin{figure}[H]
\centering
\begin{tikzpicture}
\begin{scope}[scale=5]
\filldraw[fill=yellow, pattern=north west lines, pattern color= gray, draw=white] (1,0) arc (0:180:0.5)--(0,1)--(1,1)--cycle;
\draw(-1.4,0)--(1.2,0);
\draw(0,0)--(0,1);
\node[below]() at (0,0){$0$};
\draw(1,0)--(1,1);
\node[below]() at (1,0){$1$};
\draw(-1,0)--(-1,1);
\node[below]() at (-1,0){$-1$};
\draw[] (1,0) arc (0:180:0.5); 
\draw[] (0,0) arc (0:180:0.5); 
\draw[](1,0) arc (0:180:0.25);
\draw[](0.5,0) arc (0:180:0.25);
\draw[](0,0) arc (0:180:0.25);
\draw[](-0.5,0) arc (0:180:0.25);
\draw[](1/3,0) arc (0:180:1/6);
\draw[](1/2,0) arc (0:180:1/12);
\draw[](1/2+1/3+1/6,0) arc (0:180:1/6);
\draw[](1/2+1/6,0) arc (0:180:1/12);
\draw[](1/3-1,0) arc (0:180:1/6);
\draw[](1/2-1,0) arc (0:180:1/12);
\draw[](1/2+1/3+1/6-1,0) arc (0:180:1/6);
\draw[](1/2+1/6-1,0) arc (0:180:1/12);
\draw[](1/4,0) arc (0:180:1/8);
\draw[](1/3,0) arc (0:180:1/24);
\draw[](2/5,0) arc (0:180:1/30);
\draw[](1/2,0) arc (0:180:1/20);
\draw[](12/20,0) arc (0:180:1/20);
\draw[](2/3,0) arc (0:180:1/30);
\draw[](3/4,0) arc (0:180: 1/24);
\draw[](1,0) arc (0:180:1/8);
\begin{scope}[shift={(-1,0)}]
\draw[](1/4,0) arc (0:180:1/8);
\draw[](1/3,0) arc (0:180:1/24);
\draw[](2/5,0) arc (0:180:1/30);
\draw[](1/2,0) arc (0:180:1/20);
\draw[](12/20,0) arc (0:180:1/20);
\draw[](2/3,0) arc (0:180:1/30);
\draw[](3/4,0) arc (0:180: 1/24);
\draw[](1,0) arc (0:180:1/8);
\end{scope}
\draw[](-1,0) arc (0:70:1/2);
\draw[](-1,0) arc (0:90:1/4);
\draw[](-1,0) arc (0:180:1/6);
\draw[](-1,0) arc (0:180:1/8);
\draw[](-1-1/4,0) arc (0:180:1/24);
\draw[thick](0.8,0) arc (0:180:0.95);
\node[below]() at (0.8,0){$l_+$};
\node[below]() at (-1.1,0){$l_-$};
\node[above]() at (0,0.95){$y_l$};
\node[circle, draw, inner sep=1pt, fill=black]() at (0,0.935){};
\node[]() at (1/4,0.95){$L$};
\node[]() at (0.8,0.33){$R$};
\node[]() at (0.83,0.20){$R$};
\node[]() at (-1/2,0.97){$R$};
\node[]() at (-1.08,0.4){$L$};
\end{scope}
\end{tikzpicture}
\caption{Farey tesselation and coding of a geodesic} 
\end{figure}
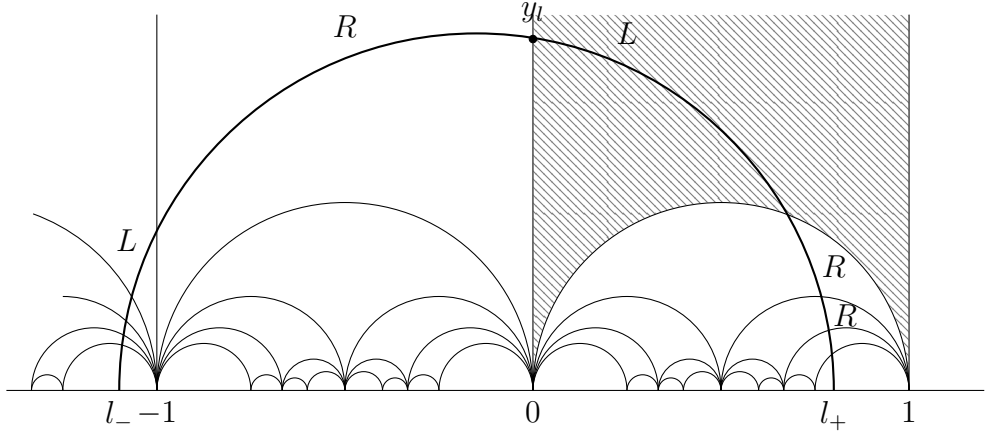 
This coding is related to the continued fraction expansion  of the endpoints $l_+$ and $l_-$ by 
\begin{align*}
l_- &= [n_{-1}, n_{-2}, ... ]^{-1} \textrm{ and }  l_+ = -[n_1,n_2,...] \textrm{ if } l_- \geq 1 ,\\ 
l_- &= -[n_{-1}, n_{-2}, ... ]^{-1} \textrm{ and }  l_+  = [n_1,n_2,...] \textrm{ if } l_- \leq -1.
\end{align*}
We now consider the generators $S$ and $T$ of $\PSL_2(\bZ)$ given by 
\[ S = \begin{pmatrix} 0 & -1 \\ 1 & 0 \end{pmatrix} \;\;\;\; \textrm{and} \;\;\;\; T = \begin{pmatrix} 1 & 1 \\ 0 & 1 \end{pmatrix} \]
which can also be thought of as their actions on $\bH$ as  $S: z \mapsto -1/z$ and $T: z \mapsto z+1$ and define the map $\tilde{\mathcal{P}} : \mathcal{L} \rightarrow \mathcal{L}$ 
\begin{align*}
\tilde{\mathcal{P}} (l) = \begin{cases} ST^{-n_1}(l) = (- [n_2, n_3,...]^{-1},[n_1, n_{-1},...])  & \textrm{if } l = ( [n_1, n_2,...]^{-1}, - [n_{-1}, n_{-2}, ...]) \\ST^{n_1}(l) = ( [n_2, n_3,...]^{-1},- [n_1, n_{-1},...],)  & \textrm{if } l = ( -[n_1, n_2,...]^{-1},  [n_{-1}, n_{-2}, ...]) \end{cases} .
\end{align*}
Let $\mathcal{P}$ be the restriction of $\tilde{\mathcal{P}}$ to the first coordinate. Then the map 
\begin{align*}
\mathcal{G} : [-1,1] &\rightarrow [-1,1] \\
x &\mapsto S \mathcal{P}S(x)
\end{align*}
is called the twisted Gauss map. It is related to the shift map $T$ by 
\[ \mathcal{G}(x) = -\textrm{sign}(x) T(|x|) .\] 
We define the shift space to be 
\[ \Sigma_* = \left\{ (x_1,x_2,...) \in (\bZ^{\times})^N | x_ix_{i+1} <0 \forall i \in \bN \right\} \]
with the shift map $\sigma_*(x_1,x_2,...) = (x_2,x_3,...)$. The map 
\begin{align*}
\rho: \Sigma_* &\rightarrow \mathcal{I} \\
(x_1,x_2,...) &\mapsto -\textrm{sign}(x_1) [ |x_1|, |x_2|, ... ] 
\end{align*}
where $\mathcal{I} = [-1,1] \cap (\bR \diagdown \bQ)$, is a bijection with the property $\rho \circ \sigma_* = \mathcal{G} \circ \rho$. 
\newline
\newline
We also wish to consider a generalization of this setup where $G$ is a modular subgroup of $\PSL_2(\bZ)$. Let $E_G$ be a set of fixed representative elements of the left cosets in $G \backslash \PSL_2(\bZ)$. We now consider the set of oriented geodesics given by 
\[ \mathcal{L}_G = \bigcup_{e \in E_G} e(\mathcal{L})\]
and the space 
\[ \overline{\Sigma}_G = \bigcup_{e \in E_G} e(\mathcal{I}) \times \{e\} \]
with the topology inherited from $\bR$. The $G$-twisted Gauss map is 
\begin{align*}
\mathcal{G}_G : \overline{\Sigma}_G &\rightarrow \overline{\Sigma}_G\\
(x,e) &\mapsto \left( eS\mathcal{P}Se^{-1}(x), e\right) 
\end{align*}
for $x \in e(\mathcal{I})$. It is shown in \cite{KS} that a certain proper shift space $\Sigma_G$ is isomorphic to $\overline{\Sigma}_G$. This shift space is
\[ \Sigma_G = \{ ((x_1,e_1),(x_2,e_2),...) \in (\bZ^{\times} \times E_G)^N | (x_1,x_2,...)\in \Sigma_*, \textrm{ and }e_{k+1}=\tau_{x_k}(e_k) \forall k \in \bN \} \]
where $\tau_{x_k}: E_G \rightarrow E_G$ is defined by 
\[ \tau_{x_k} (e_k) \equiv_G e_kST^{x_k}\] 
equipped with the shift map 
\begin{align*}
\sigma: \Sigma &\rightarrow \Sigma \\
((x_1,e_1),(x_2,e_2),...) &\mapsto ((x_2,e_2),(x_3,e_3),...) 
\end{align*}
and metric 
\[d(((x_k,e_k))_k, ((x'_k,e_k')_k)) = \sum_{i-=1}^{\infty}\frac{1}{2^i} \left( 1 - \delta_{(x_i,e_i),(x'_i,e'_i)}\right). \]
The isomorphism $\overline{\Sigma}_G \rightarrow \Sigma_G$ is given by 
\[ (e (\pm[n_1,n_2,...]), e) \mapsto ((\mp n_1,e),(\mp n_2, \tau_{\mp n_1}(e),(\mp n_3, \tau_{\mp n_2}(\tau_{\mp n_1}(e))),...). \]

Formulating the limiting modular symbols in terms of the shift space rather than directly in terms of points in $\bR$ is useful because the shift space $(\Sigma_G, \sigma)$ is known to be finitely irreducible (Prop 3.1 \cite{KS}). This means that there is a finite set $W \subset \Sigma_G^*$, where $\Sigma_G^*$ is the set of finite admissible words in the alphabet $\bZ^{\times} \times E_G$, such that for any $a,b \in \bZ^{\times} \times E_G$ there exists $w \in W$ such that $awb \in \Sigma_G^*$.

It is also shown in \cite{KS} that, as elements in $((x_i,e_i))_i \in \Sigma_G$ satisfy $e_{k+1}=\tau_{x_k}(e_k)$, there is a relation in terms of the continued fraction expansion of $x =-\textrm{sign}(x_1) [|x_1|, |x_2|, ... ] = [\tilde{x_1}, \tilde{x_2},...]$
\[ e_{k+1} \equiv_G e_1 ST^{\tilde{x_1}}...ST^{\tilde{x_k}} = e_1 \overline{g}_k(x) \]
where 
\begin{equation}\label{linegk}
\overline{g}_k(x) = \begin{pmatrix} -\textrm{sign}(x_1) p_{k-1}(|x|) & (-1)^kp_k(|x|) \\ q_{k-1}(|x|) & (-1)^{k+1} \textrm{sign}(x_1)q_k(|x|)\end{pmatrix}.
\end{equation}
Similar to equation \ref{nmaction}, we have the relation describing the action of $\overline{g}_k(x)$ on higher-weight modular symbols 
  \begin{align}\label{nmaction2}
& \overline{g}_k(x) | \{0, i\infty, n, m\}_G \nonumber \\
 &    = - \left\{ -\textrm{sign}(x_1)  \frac{p_{k-1}(|x|)}{q_{k-1}(|x|)}, -\textrm{sign}(x_1) \frac{p_k(|x|)}{q_k(|x|)} ,  \begin{pmatrix} 0& 1 \\ -1 & -|x_k| \end{pmatrix} \dots \begin{pmatrix}0& 1 \\ -1 & -|x_1| \end{pmatrix} \begin{pmatrix}n \\ m \end{pmatrix}\right\}_G . 
    \end{align}
\subsection{Limiting modular symbol for the shift space} 
We now define a corresponding modular symbol on the shift space. Let $\overline{X}_G = (\bH \cup P^1(\bQ) ) / G$. For an element of $\Sigma_G$, the associated limiting modular symbol on the shift space is
\begin{equation} \label{kugalimmod}
\begin{aligned}
\tilde{l}_G : \Sigma_G &\rightarrow H_1(\overline{X}_G, \bR) \\
((x_k, e_k))_k &\mapsto \lim_{t \rightarrow \infty} \frac{1}{t} \{i, e_1(x + i \exp(-t))\}_G 
\end{aligned}
\end{equation}
where we set $x = -\textrm{sign}(x_1) [|x_1|, |x_2|, ... ] \in \mathcal{I}$. 
The modular symbol on the right-hand side, $\{i, e_1(x + i \exp(-t))\}_G \in H_1(X_G, \bR)$, is the standard modular symbol. 

\begin{figure}
\begin{tikzpicture}[scale=1.3]
\draw (3,0) arc(0:143:1.67); 
\draw[white, thick] (3,0) arc(0:90:1.67); 
\draw (-5,0) -- (5,0); 
\node[below] at (3,0) {$x$}; 
\draw[] (3,0) -- (3,3); 
\draw(0,0) --(0,3); 
\node[right] at (3, 2.5) {\small{$\gamma(t) = x +  i \exp(-t))$}}; 
\node at (0, 1) [circle,fill,inner sep=1pt]{};

\node[below] at (2,0) {\small{$e_1(x)$}}; 
\draw[thick] (2,0) arc (0:180:2.5); 
\node[left] at (0,1) {$i$}; 
\draw[dashed] (2,0) arc(0:125:1.25); 
\draw (1.61, 1.34) arc(67.4:135:1.45); 
\node at (1.61, 1.34) [circle,fill,inner sep=1pt]{};
\node[right] at (1.61, 1.34) []{$e_1(\gamma(t_2))$};
\node at (1.37, 1.67) [circle,fill,inner sep=1pt]{};
\node[right] at (1.37, 1.67) []{$e_1(\gamma(t_1))$};
\node at (-2.5,2.3) {\small{$e_1(\gamma(t))$}}; 
\node[below] at (-3,0) {\small{$e_1(i \infty)$}};

\node at (3, 1) [circle,fill,inner sep=1pt]{};
\node[right] at (3, 1) []{$\gamma(t_2)$};
\node at (3, 1.4) [circle,fill,inner sep=1pt]{};
\node[right] at (3, 1.4) []{$\gamma(t_1)$};

\end{tikzpicture}
\caption{Definition of the limiting modular symbol for the shift space}
\end{figure}
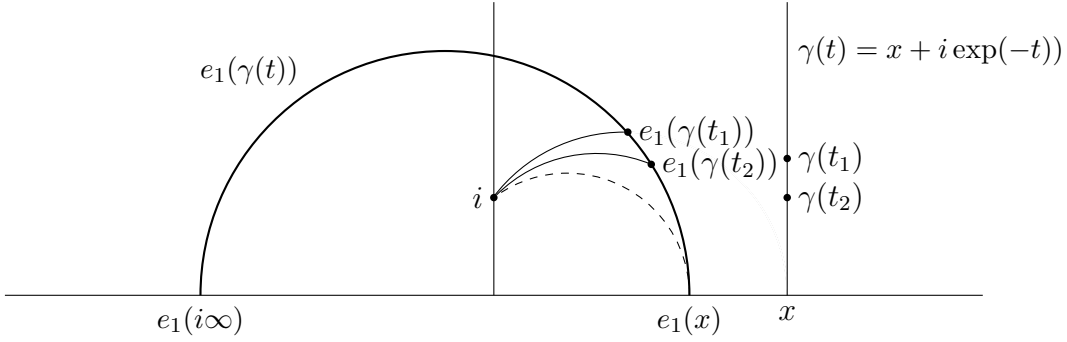

It is known that this limit can be equivalently written by approximating the point $e_1(x)$ by its continued fraction expansion (Proposition 4.2 \cite{KS})
\begin{equation} \label{prop4.2}
\tilde{l}_G(((x_k, e_k))_k) = \lim_{n \rightarrow \infty} \frac{1}{2 \log q_n(|x|)} \sum_{k=1}^{n} \{ e_k (i\infty), e_k(0) \}_G. 
\end{equation} 
We generalize this picture to the higher-weight setting by putting 
\begin{equation}
\begin{aligned}
\tilde{l}_{G,n,m} : \Sigma_G &\rightarrow H_1(\overline{X_G}, \Pi, (R_1\Phi_*\bQ)^w) \\
((x_k, e_k))_k &\mapsto \lim_{t \rightarrow \infty} \frac{1}{t} \{i, e_1(x + i \exp(-t)), n, m\}_G.
\end{aligned}
\end{equation}
We proceed by obtaining a similar result to equation \ref{prop4.2}, but now with modular symbols of higher weight. Importantly, the result holds everywhere on $\Sigma_G$. 
\begin{thm} \label{mainthm} For $((x_k,e_k))_k  \in \Sigma_G$ we have
\[ \tilde{l}_{G,N,M} \left( ((x_k,e_k))_k \right) = \lim_{n \rightarrow \infty} \frac{1}{2 \log q_n(|x|)} \sum_{k=1}^{n} \{ e_k (i\infty), e_k(0), \tilde{g}_{k-1}^{-1}(x) \cdot (N,M) \}_G \]
where $\tilde{g}_{k-1}(x) = e_1 \overline{g_{k-1}}(x) e_k^{-1}$ and we set $x = -\textrm{sign}(x_1) [|x_1|, |x_2|, ... ]$. 
\end{thm}
\begin{proof}
The proof follows the strategy outlined in \cite{KS}, but here we track the additional $(n,m)$-coordinate data of the higher-weight modular symbol. The general strategy is as follows. We begin by showing that 
\begin{equation}
L_{G,N,M} \left( ((x_k,e_k))_k \right)  := \lim_{n \rightarrow \infty} \frac{1}{2 \log q_n(|x|)} \sum_{k=1}^{n} \{ e_k (i\infty), e_k(0), \tilde{g}_{k-1}^{-1} (N,M)\}_G.
\end{equation}
exists if and only if there is a sequence $(t_n)_{n \in \bN}$ tending to infinity such that 
\begin{equation} \label{tlimit}
\lim_{n \rightarrow \infty} \frac{1}{t_n} \{i, e_1(x + i e^{-t_n}), N,M\}_G 
\end{equation} 
exists, and that if either limit exists they coincide. Then, we will show that the limit \ref{tlimit} does not depend on the particular sequence $(t_n)_{n \in \bN}$ chosen. 
\\
\\
The main idea is to write the geodesic passing through $e_1(i\infty)$ and $e_1(x)$, in terms of geodesics related to the continued fraction approximants of $x$. Define a sequence of points in $P^1(\bQ)$ by 
\begin{equation}
\begin{aligned}
\xi_1 & = e_1(i\infty) \\
\xi_n & =  e_1\left( -\textrm{sign}(x_1) \frac{p_{n-2}(|x|)}{q_{n-2}(|x|)}\right) \;\;\;\;\; n\geq 2
\end{aligned}
\end{equation}
and let $\omega_n$ be the oriented geodesic in $\bH \cup P^1(\bQ)$ which starts at $\xi_n$ and ends at $\xi_{n+1}$. 
\\
\\
Next, let $l(x)$ be the oriented vertical geodesic running from $i \infty$ to $x$  and let $e_1(l(x))$ be its image. The image $e_1(l(x))$ is a geodesic starting at $\xi_1$ and ending at $e_1(x)$. Define a sequence of points along $e_1(l(x))$ by 
\begin{equation}
\begin{aligned}
y_n = \omega_n \cup e_1(l(x)).
\end{aligned}
\end{equation}
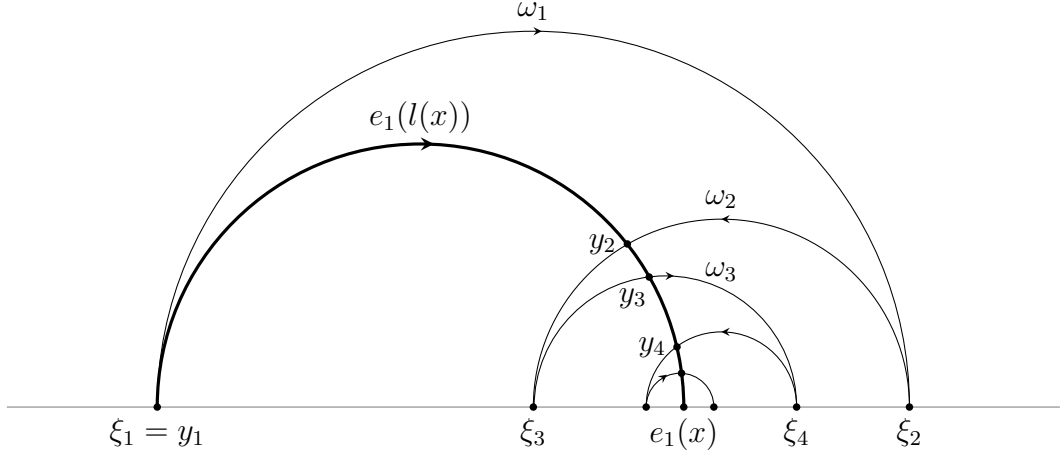
\begin{figure}[H]
\centering
\begin{tikzpicture}[scale=10,>=stealth]
\draw [color=gray](-0.2,0) -- (1.2,0);

\node at (0,0) [circle,fill,inner sep=1pt]{};
\node at (0,0) [below]{$\xi_1 = y_1$};
\node at (1,0) [circle,fill,inner sep=1pt]{};
\node at (1,0) [below]{$\xi_2$};
\node at (1/2,0) [circle,fill,inner sep=1pt]{};
\node at (1/2,0) [below]{$\xi_3$};
\node at (0.85,0) [circle,fill,inner sep=1pt]{};
\node at (0.85,0) [below]{$\xi_4$};
\node at (0.65,0) [circle,fill,inner sep=1pt]{};
\node at (0.74,0) [circle,fill,inner sep=1pt]{};
\node at (0.7,0) [circle,fill,inner sep=1pt]{};
\node at (0.7,0) [below]{$e_1(x)$};

\node at (0.697,0.045) [circle,fill,inner sep=1pt]{};

\node at (0.691,0.08) [circle,fill,inner sep=1pt]{};
\node at (0.691,0.08) [left]{$y_4$};
\node at (0.654,0.173) [circle,fill,inner sep=1pt]{};
\node at (0.665,0.173) [below left]{$y_3$};
\node at (0.625,0.217) [circle,fill,inner sep=1pt]{};
\node at (0.625,0.217) [left]{$y_2$};

\node at (0.35,0.35)[above]{$e_1(l(x))$};
\node at (0.5,0.5)[above]{$\omega_1$};
\node at (0.75,0.25)[above]{$\omega_2$};
\node at (0.75,0.18)[]{$\omega_3$};

\draw[decoration={
    markings,
    mark=at position 0.5 with {\arrow{<}}}, postaction={decorate}](1,0) arc (0:180:0.5);
\draw[decoration={
    markings,
    mark=at position 0.5 with {\arrow{>}}}, postaction={decorate}](1,0) arc (0:180:0.25);
\draw[decoration={
    markings,
    mark=at position 0.5 with {\arrow{<}}}, postaction={decorate}](0.85,0) arc (0:180:0.1745);
    \draw[decoration={
    markings,
    mark=at position 0.5 with {\arrow{>}}}, postaction={decorate}](0.85,0) arc (0:180:0.1);
    \draw[decoration={
    markings,
    mark=at position 0.7 with {\arrow{<}}}, postaction={decorate}](0.74,0) arc (0:180:0.045);
\draw[decoration={
    markings,
    mark=at position 0.5 with {\arrow{<}}}, postaction={decorate}, very thick](0.7,0) arc (0:180:0.35);
\end{tikzpicture}
\caption{Approximating $e_1(l(x))$ by continued fractions}
\end{figure}
Note that the oriented geodesic path from $y_n$ to $y_{n+1}$ is homologous to the geodesic path running from $y_n$ to $\xi_{n+1}$ to $y_{n+1}$. Therefore, we have that 
\[\{y_n, y_{n+1}, N,M\}_G = \{y_n, \xi_{n+1}, N,M\}_G  + \{\xi_{n+1}, y_{n+1}, N,M\}_G  \]
for all $n \in \bN$. 
\\
\\
Recall that we have 
\[ e_{n+1} \equiv_G e_1\overline{g}_n(x)\] 
where $\overline{g}_n$ for $n\geq2$ is defined in \ref{linegk} and $\overline{g_1} = \textrm{id}$. Therefore, there exists some $\tilde{g}_n(x) \in G$ such that 
\[ \tilde{g}_n(x) e_{n+1} = e_1\overline{g}_n(x). \]
By property \ref{Ginv} of the modular symbols, and by directly acting $g_n$ on the points $0$ and $i\infty$ by fractional linear transformations, we get
\begin{align*}
\{ e_n(i\infty), e_n(0), \tilde{g}_{n-1}^{-1}(x) \cdot (N,M) \}_G &= \{\tilde{g}_{n-1}(x) e_{n}(i\infty), \tilde{g}_{n-1}(x)e_n(0), N,M\}_G \\
&= \{ e_1\overline{g}_{n-1}(x)(i\infty), e_1 \tilde{g}_{n-1}(x)(0), N,M \}_G \\
& = \{ \xi_n, \xi_{n+1}, N,M \}_G. 
\end{align*}
Using homologous paths and additivity of the modular symbols we find 
\begin{align*}
& \{i, y_{n+1}, N,M\}_G  = \{i, y_2,N,M\}_G + \{ y_2, y_{n+1}, N,M\}_G \\
&= \{i, y_2,N,M\}_G + \sum_{k=2}^n \{y_k, y_{k+1},N,M\}_G \\
& = \{i, y_2,N,M\}_G + \sum_{k=2}^n \left( \{y_k, \xi_{k+1},N,M\}_G + \{\xi_{k+1}, y_{k+1},N,M\}_G \right)\\
& = \{i, y_2,N,M\}_G - \{ \xi_2, y_2, N,M\}_G - \{y_{n+1}, \xi_{n+1},N,M\}_G + \sum_{k=2}^{n+1} \{\xi_k, \xi_{k+1}, N,M\}_G\\
& = \{i, \xi_1,N,M\}_G - \{y_{n+1}, \xi_{n+1},N,M\}_G+ \sum_{k=1}^{n+1} \{\xi_k, \xi_{k+1}, N,M\}_G\\
& = \{i, \xi_1,N,M\}_G - \{y_{n+1}, \xi_{n+1},N,M\}_G+ \sum_{k=1}^{n+1} \{e_k(i\infty), e_k(0),  \tilde{g}_{k-1}^{-1} \cdot  (N,M) \}_G
\end{align*}
Let the sequence $t_n$ be defined by the equation 
\[e_1(x + i e^{-t_n}) := y_n.\] 
An argument involving hyperbolic geometry gives an estimate $e^{-t_n} \sim (q_n(|x|))^2$, for sufficiently large $n$ (see Lemma 3.3 of \cite{KS2}). 
\\
\\ 
With this we can complete the first part of the proof, concluding the equivalence of the limits: 
\begin{align*}
L_{G,N,M} \left( ((x_k, e_k))_k \right) & = \lim_{n \rightarrow \infty} \frac{1}{2 \log q_n(|x|)} \sum_{k=1}^{\infty} \{ e_k (i\infty),  e_k(0), \tilde{g}_{k-1}^{-1} \cdot  (N,M) \}_G \\
& = \lim_{n \rightarrow \infty} \frac{1}{t_n} ( \{ i,y_{n+1},N,M\}_G  + \{y_{n+1}, \xi_{n+1},N,M\}_G - \{i, \xi_1, N,M\}_G ) \\
& = \lim_{n \rightarrow \infty} \frac{1}{t_n} ( \{ i,y_{n},N,M\}_G  + \{y_{n}, \xi_{n+1},N,M\}_G - \{i, \xi_1, N,M\}_G ) \\
& = \lim_{n \rightarrow \infty} \frac{1}{t_n}  \{ i,y_{n},N,M\}_G \\
& = \lim_{n \rightarrow \infty} \frac{1}{t_n}  \{i, e_1(x + i e^{-t_n}), N, M\}_G
\end{align*}
The second step of the proof is to show that the limit $\lim_{n \rightarrow \infty} \frac{1}{t_n}  \{i, e_1(x + i e^{-t_n}), N, M\}_G$ is independent of the choice of sequence $(t_n)$ tending to infinity. Recall that we have a non-degenerate pairing between the higher-weight modular symbols and the space of cusp forms $S_{w+2}(G) \oplus \overline{S_{w+2}(G) }$. Suppose that $L_{G,N,M} \left( ((x_k,e_k))_k \right)$ exists. For $\Phi = (\Phi_1, \Phi_2) \in  S_{w+2}(G) \oplus \overline{S_{w+2}(G) }$ arbitrary and $t>0$, let 
\[ \alpha_{\Phi} = \langle L_{G,N,M} \left( ((x_k,e_k))_k \right) , \Phi \rangle \]
and let 
\[n_t = \sup \{ n \in \bN : 2 \log q_n(|x|) \leq t \} .\]
Our aim is to show that for all $\Phi \in S_{w+2}(G) \oplus \overline{S_{w+2}}(G)$, 
\[ \limsup_{t \rightarrow \infty} \left| \frac{\langle \{i, e_1(x + i e^{-t}),N,M\}_G, \Phi \rangle}{t} - \frac{\langle \sum_{k=1}^{n_t} \{ e_k(i\infty), e_k(0), \tilde{g}_k^{-1} (N,M)\}_G, \Phi \rangle}{2\log q_{n_t}(|x|)}\right| =0,  \]
which will allow us to conclude that $\tilde{l}_{G,N,M}\left( ((x_k,e_k))_k \right)$ exists and is equal to $L_{G,N,M} \left( ((x_k,e_k))_k \right)$. We obtain a bound following exactly the same strategy as \cite{KS}, but we repeat it here for completeness. 
\begin{align*}
& \limsup_{t \rightarrow \infty}  \left| \frac{\langle \{i, e_1(x + i e^{-t}),N,M\}_G, \Phi \rangle}{t} - \frac{\left\langle \displaystyle\sum_{k=1}^{n_t} \{ e_k(i\infty), e_k(0), \tilde{g}_k^{-1} (N,M)\}_G, \Phi \right\rangle}{2\log q_{n_t}(|x|)}\right| \\
& = \limsup_{t \rightarrow \infty} \left| \frac{{2\log q_{n_t}(|x|)}\langle \{i, e_1(x + i e^{-t}),N,M\}_G, \Phi \rangle}{{2t\log q_{n_t}(|x|)}} - \frac{t \left\langle \displaystyle\sum_{k=1}^{n_t} \{ e_k(i\infty), e_k(0), \tilde{g}_k^{-1} (N,M)\}_G, \Phi \right\rangle}{2t\log q_{n_t}(|x|)}\right| \\
& \leq \limsup_{t \rightarrow \infty} \left| \frac{1}{t} \langle \{i, e_1(x+ie^{-t}), N, M\}_G - \displaystyle\sum_{k=1}^{n_t} \{ e_k(i\infty), e_k(0), \tilde{g}_k^{-1} (N,M)\}_G , \Phi \rangle \right| \\
& \hspace{3cm} + \limsup_{t \rightarrow \infty} \left| \frac{2\log q_{n_t}(|x|) - t}{t} \right| \left| \frac{ \left\langle \displaystyle\sum_{k=1}^{n_t} \{ e_k(i\infty), e_k(0), \tilde{g}_k^{-1} (N,M)\}_G, \Phi \right\rangle}{2\log q_{n_t}(|x|)}\right| \\
&\leq \limsup_{t \rightarrow \infty} \frac{const.}{t} + \limsup_{n \rightarrow \infty} \frac{\log |x_{n+1}|}{\log q_n(|x|)} |\alpha_{\Phi}| \\
& =  |\alpha_{\Phi}| \limsup_{n \rightarrow \infty} \frac{\log |x_{n+1}|}{\log q_n(|x|)} 
\end{align*}
where in the last bound we are using the recursion relation $q_n(|x|)  = |x_{n+1}|q_{n-1}(|x|) + q_{n-2}(|x|)$. 
\\
\\
In the case that $\alpha_{\Phi}=0$ for all $\Phi \in  S_{w+2}(G) \oplus \overline{S_{w+2}(G) }$, the result follows immediately. We will assume wlog that there is some $\Phi$ such that $\alpha_{\Phi} > 0$. In this case, we will show that $\limsup_{n \rightarrow \infty}  \frac{\log |x_{n+1}|}{\log q_n(|x|)}=0.$ To do this, we first observe that 
\begin{align*}
\alpha_{\Phi} &= \lim_{n \rightarrow \infty} \frac{\left\langle \sum\limits_{k=1}^{n+1} \{e_k(i\infty),e_k(0),\tilde{g}_k^{-1} (N,M) \}_G, \Phi \right\rangle}{2\log q_{n+1}(|x|)} \\
&= \lim_{n \rightarrow \infty} \frac{\left\langle \sum\limits_{k=1}^{n} \{e_k(i\infty),e_k(0),\tilde{g}_k^{-1} (N,M) \}_G + \{e_{n+1}(i\infty),e_{n+1}(0),\tilde{g}_{n+1}^{-1} (N,M) \}_G, \Phi \right\rangle}{2\log q_{n}(|x|)+ 2\log |x_{n+1}| }\\
&= \lim_{n \rightarrow \infty} \frac{\left\langle \sum\limits_{k=1}^n \{e_k(i\infty, e_k(0), \tilde{g}_k^{-1} (N, M) \}_G, \Phi \right\rangle \left( 1 + \frac{\langle\{e_{n+1}(i\infty),e_{n+1}(0),\tilde{g}_{n+1}^{-1} (N,M) \}_G, \Phi \rangle}{\left\langle \sum\limits_{k=1}^n \{e_k(i\infty, e_k(0), \tilde{g}_k^{-1} (N, M) \}_G, \Phi \right\rangle}\right) }{2\log q_n(|x|)\left( 1 + \frac{\log|x_{n+1}|}{\log q_n(|x|)} \right)} \\
& = \alpha_{\Phi}  \lim_{n \rightarrow \infty} \frac{ 1 + \frac{\langle \{e_{n+1}(i\infty),e_{n+1}(0),\tilde{g}_{n+1}^{-1} (N,M) \}_G, \Phi \rangle}{\left\langle \sum\limits_{k=1}^n \{e_k(i\infty), e_k(0), \tilde{g}_k^{-1} (N, M) \}_G, \Phi \right\rangle} }{ 1 + \frac{\log|x_{n+1}|}{\log q_n(|x|)} } .
\end{align*}
Suppose for contradiction that $\limsup_{n \rightarrow \infty} \frac{\log|x_{n+1}|}{\log q_n(|x|)}>0$. Then there is a subsequence $(n_k)_k$ such that $\lim_{k \rightarrow \infty} \frac{\log|x_{n_k+1}|}{\log q_{n_k}(|x|)} > 0$, and hence $\lim_{k \rightarrow \infty} |x_{n_k +1}| = \infty$. Since we have assumed that $\alpha_{\Phi}>0$, we get
\begin{align*}
1 &= \lim_{k \rightarrow \infty} \frac{ 1 + \frac{\langle \{e_{n_k+1}(i\infty),e_{n_k+1}(0),\tilde{g}_{n_k+1}^{-1} (N,M) \}_G, \Phi \rangle}{\left\langle \sum\limits_{j=1}^{n_k} \{e_j(i\infty), e_j(0), \tilde{g}_j^{-1} (N, M) \}_G, \Phi \right\rangle} }{ 1 + \frac{\log|x_{n_k+1}|}{\log q_{n_k}(|x|)} }  \\
&=  \lim_{k \rightarrow \infty} \frac{\log q_{n_k}(|x|)}{\left\langle \sum\limits_{j=1}^{n_k} \{e_j(i\infty), e_j(0), \tilde{g}_j^{-1} (N, M) \}_G, \Phi \right\rangle} \frac{\langle \{e_{n_k+1}(i\infty),e_{n_k+1}(0),\tilde{g}_{n_k+1}^{-1} (N,M) \}_G, \Phi \rangle}{\log|x_{n_k+1}|} \\
& = \frac{1}{\alpha_{\Phi}} (0) = 0
\end{align*}
This is a contradiction, so we conclude that $\limsup_{n \rightarrow \infty} \frac{\log|x_{n+1}|}{\log q_n(|x|)}=0$. 
\end{proof}

As in the weight-$2$ case, one can conclude from Theorem \ref{mainthm} that the limiting modular symbol is non-zero at the quadratic irrationalities, which have repeating continued fraction expansion. Having defined the higher-weight limiting modular symbol, the next step in this research program is to modify the boundary-$\GL_2$ system of \cite{MarPan} to the higher-weight setting, obtain the ground states, and study their evaluations on the arithmetic subalgebra. The author intends to update this note with further results in this direction. 

\subsection*{Acknowledgments} The author would like to thank Matilde Marcolli for helpful discussions.

\end{document}